\documentclass[11pt]{article}
\usepackage{amssymb,url,amsmath,graphicx, amsthm}
\newtheorem{theorem}{Theorem}[section]

\newtheorem{lemma}{Lemma}[section]

\newcommand{\hdiv}{H(\mathrm{div})}
\newcommand{\norm}[1]{\left\| #1 \right\|}
\newcommand{\seminorm}[1]{\left| #1 \right|}
\newcommand{\weightednorm}[2]{\norm{#1}_{#2}}
\newcommand{\weightedseminorm}[2]{\seminorm{#1}_{#2}}

\title{Mixed finite elements for global tide models with nonlinear damping}
\date{2 June 2017}
\author{
  Colin J. Cotter\thanks{Imperial College London, South Kensington Campus;
London SW7 2AZ.  CJC acknowledges support from NERC grant NE/I016007/1.
  }
  P. Jameson Graber\thanks{Department of Mathematics, Baylor University; One Bear Place \#97328; Waco, TX 76798-7328; USA.  PJG acknowledges support from NSF grant 1612880.}
  Robert C. Kirby\thanks{Department of Mathematics, Baylor University; One Bear Place \#97328; Waco, TX 76798-7328; USA.  RCK acknowledges support from NSF grant 1525697.}}
\begin{document}
\maketitle
\begin{abstract}
We study mixed finite element methods for the rotating
shallow water equations with linearized momentum terms but nonlinear drag.
By means of an equivalent second-order formulation, we prove long-time
stability of the system without energy accumulation.  We also give rates of damping in unforced systems and various
continuous dependence results on initial conditions and forcing
terms.
\emph{A priori} error estimates for the momentum and free
surface elevation are given in $L^2$ as well as for the time
derivative and divergence of the  momentum.  Numerical
results confirm the theoretical results regarding both energy damping
and convergence rates.
\end{abstract}

\section{Introduction}

Accurate modeling of tides is important in several scientific
disciplines.  Tides' strong impact on sediment transport and coastal
flooding makes them of interest to geologists.  Oceanographers suggest
that breaking internal tides provide a mechanism for vertical mixing
of temperature and salinity that might sustain the global ocean
circulation \cite{GaKu2007,MuWu1998}. To predict the global tides away
from coastlines, it is often sufficient to model the barotropic tide
using the rotating shallow water equations. In the open ocean, the
nonlinear advection terms have an insignificant effect on the
barotropic tide, and many models consist of the linear rotating
shallow-water equations with a parameterised drag term to model the
effects of bottom friction~\cite{stammer2014accuracy}. In
\cite{Hi_etal2011}, a linear model similar to this was solved globally
to produce boundary conditions for a more sophisticated local model.
The barotropic model can be made more sophisticated by adding
additional dissipative terms that model other dissipative mechanisms
in the barotropic tide, due to baroclinic tides, for example
\cite{JaSt2001}.

The possibility of unstructured triangular meshes make finite element
methods attractive for modelling the world's oceans, including irregulary coastlines and topography~\cite{We_etal2010}. 
Recent years have seen much discussion about mixed
finite element pairs to use as the horizontal discretization for
atmosphere and ocean models.  In papers such as~\cite{CoLaReLe2010,CoHa2011,Da2010,RoRoPo2007,RoRo2008,Ro2005,Ro2012,RoBe2009},
we see many details regarding the numerical dispersion relations obtained when
discretizing the rotating shallow water equations. Then, in~\cite{CoKi},
we took a different angle, studying the behavior of
discretizations of the forced-dissipative rotating shallow-water
equations used for predicting global barotropic tides.  In particular,
energy techniques were used to show that discrete solutions approach
the correct long-time solution in response to quasi-periodic forcing.
Since the linearized energy only controls the divergent part of the
solution, we chose finite element spaces for which there is 
a natural discrete Helmholtz decomposition and such that
the Coriolis term projects the divergent and divergence-free
components of vector fields correctly onto each other.
Hence, we used compatible, finite element
spaces (\emph{i.e.} those which arise naturally from the finite
element exterior calculus \cite{arnold2006finite}), first proposed
for numerical weather prediction in \cite{CoSh2012} and then extended
to develop finite element methods for the nonlinear rotating
shallow-water equations on the sphere that can conserve energy,
enstrophy and potential
vorticity~\cite{CoTh2014,McCo2014,RoHaCoMc2013}.  In~\cite{CoKi}, the 
discrete Helmholtz decomposition allowed us to show that
mixed finite  element discretizations of the forced-dissipative linear
rotating shallow-water equations have the correct long-time energy
behavior, and the linear nature of the equations also led to natural
optimal \emph{a priori} error estimates.

Finite element methods' ability to use unstructured grids also allows
coupling of global tide structure with local coastal dynamics. Both
discontinuous Galerkin \cite{salehipour2013higher} and continuous finite element
approaches \cite[for
  example]{FoHeWaBa1993,kawahara1978periodic,Le_etal2002} have been advocated and successfully used. The $RT_0-P_0$ (lowest order Raviart-Thomas element for 
velocity and piecewise constant for height)
was proposed for coastal tidal modeling in
\cite{walters2005coastal}; this example is included in the family of 
discretisations that we consider here.

In~\cite{CoKi}, we restricted attention to the linear bottom drag
model as originally proposed in \cite{La45}.  Quadratic damping laws
are more realistic and are what is used in barotropic tide predictions
\cite{stammer2014accuracy}, but the nonlinearity means that one cannot
simply apply a Fourier transform in time and solve for each mode
separately.  It is assumed that the system has some kind of
time-dependent attracting solution under the quasi-periodic tidal
forcing, to which all solutions converge as $t\to \infty$. Calculating
this attracting solution is the goal of barotropic tide modelling.
Then, one can solve the equations in the time domain until this
attracting solution is reached (``spun up'').  Alternatively,
\cite{Hi_etal2011} proposed an iterative method for approximating this
attracting solution by solving for pure time-periodic solutions at
different tidal frequencies, and feeding the solutions back via the
nonlinearity. In this paper, we concentrate on the former aspect,
i.e. showing that the numerical discretisation has an attracting
solution and whether this attracting solution converges to the true
attracting solution as the resolution is refined.

The nonlinearity also presents significant difficulties to the
analysis, even though it is much more benign than the advective
nonlinearity in the full equation set.  In this paper, we extend our
work in the linear case by adapting techniques from the nonlinear PDE
literature (see especially
\cite{cavalcanti2007well,lasiecka1993uniform} and references therein)
to the finite element setting.  We consider a family of damping laws
that are nonlinear for small velocity but behave linearly for large
velocity.  We require monotonicity and some other technical
assumptions on the on the nonlinearity, and these include the
quadratic case and other power laws.  As an alternative to modifying
the damping term for large velocity, \emph{a priori} assumptions (or
better, estimates) on the size of solutions would allow us to use an
unmodified law.  At any rate, provided that the velocity in fact
remains bounded, one may compute with the unmodified (i.~.e.~not
forced to be linear at infinity) law.  As with the linear case, we
believe that the applicability of our work is not limited to the
shallow water case, but to other nonlinearly damped hyperbolic systems
for which the appropriate function spaces have discrete Helmholtz
decompositions, such as damped electromagnetics or elastodynamics.

In addition to mixed finite elements' application to tidal
models in the geophysical literature, this work also builds on 
existing literature for mixed discretization of the acoustic
equations. The first such investigation is due to
Geveci~\cite{geveci1988application}, where exact energy conservation
and optimal error estimates are given for the semidiscrete first-order
form of the model wave equation.  Later
analysis~\cite{cowsar1990priori,jenkins2003priori} considers a second 
order in time wave equation with an auxillary flux at each time step.
In~\cite{kirbykieu}, Kirby and Kieu return to the first-order
formulation, giving additional estimates
beyond~\cite{geveci1988application} and also analyzing the symplectic
Euler method for time discretization.  From the standpoint of this
literature, our model appends additional terms for
the Coriolis force and damping to the simple acoustic model.
We restrict ourselves to semidiscrete analysis in this work, but pay
careful attention the extra terms in our estimates, showing how study
of an equivalent second-order equation in $\hdiv$ proves proper
long-term behavior of the model.

In the rest of the paper, we describe the tidal model and a general
finite element discretization in Section~\ref{se:model}.
Section~\ref{se:energy} gives the three major results of this paper.
In particular, we show that for any initial data and forcing function
with a uniform time bound, the system energy also remains uniformly
bounded.  Then, we give two
continuous dependence results.  The first of these works with solutions
corresponding to identical forcing but different initial data.  In
this case, we show that the energy of the difference tends to zero
over time at a rate that depends on the particular nonlinearity.  As
corollaries of this, we obtain the existence of global attracting
solutions and also effective energy decay rates for the unforced
system.  Our second dependence result allows both the initial data and
forcing to vary, when the energy difference is bounded unformly in time by
the sum of a term that is linear in the initial energy perturbation
and nonlinear in the forcing perturbation.  In Section~\ref{se:error},
we give two kinds of \emph{a priori} error estimates.  The first,
using standard techniques, shows that the error is
optimal with the power of $h$, but the constant degrades exponentially
in time.  The second applies the continuous dependence result of
Section~\ref{se:energy} to give estimates with a generically
suboptimal power of $h$, but that hold \emph{uniformly} for all time.
Finally, we present some numerical experiments in Section~\ref{sec:num}.
As a note, our previous work~\cite{CoKi} in the linear case included
application of the techniques in~\cite{holst2012geometric} when the
domain is actually a more general manifold.  We do not include this
extension here, but the nonlinear 
should not include additional complications.

\section{Description of finite element tidal model}
\label{se:model}
We start with the nondimensional linearized rotating shallow water
model with linear forcing and a possibly nonlinear drag term on a two
dimensional surface $\Omega$, given by
\begin{equation}
\begin{split}
u_t + \frac{f}{\epsilon} u^\perp + \frac{\beta}{\epsilon^2} \nabla
\left( \eta - \eta^\prime \right) + g(u) & = 0, \\
\eta_t + \nabla \cdot \left( H u \right) & = 0,
\end{split}
\end{equation}
where $u$ is the nondimensional two dimensional velocity field tangent
to $\Omega$, $u^\perp=(-u_2,u_1)$ is the velocity rotated by $\pi/2$, $\eta$ is the nondimensional free surface elevation above
the height at state of rest, $\nabla\eta'$ is the (spatially varying)
tidal forcing, $\epsilon$ is the Rossby number (which is small for
global tides), $f$ is the spatially-dependent non-dimensional Coriolis
parameter which is equal to the sine of the latitude (or which can be
approximated by a linear or constant profile for local area models),
$\beta$ is the Burger number (which is also small),
$H$ is the (spatially varying) nondimensional fluid
depth at rest, and $\nabla$ and $\nabla\cdot$ are the intrinsic
gradient and divergence operators on the surface $\Omega$,
respectively.

The damping function $g$ is the major focus of this work.  We assume
that $g(u)$ is possibly inhomogeneous in that $g(u) = g(x, u)$,
although for simplicity we suppress the extra argument.  All bounds
given on $g$ will be assumed to hold uniformly in $x$.  Although our
main interest is a power law, we only make structural assumptions on
$g$.  At the very least, we assume
\begin{itemize}
\item Monotonicity.  For all $v$,
  \begin{equation}
    \label{eq:monot}
    g(v) \cdot v > 0.
  \end{equation}
\item Linear growth for large velocity.  There exists an $M > 0$ such that for all $|v| > 1$, we have
  \begin{equation}
    \label{eq:lineargrowth}
    |v| + |g(v)|^2 \leq M g(v) \cdot v.
  \end{equation}
\end{itemize}
These assumptions are sufficient to give long-time stability of
solutions, although the continuous dependence results will require
stronger assumptions (which still hold for $g$ of practical interest).
These will be made precise later in the paper.

We will work with a slightly generalized version of the forcing term,
which will be necessary for our later error analysis.  Instead of
assuming forcing of the form 
$\frac{\beta}{\epsilon^2} \nabla \eta^\prime$, 
we assume some $F \in L^2$, giving our model as
\begin{equation}
\begin{split}
u_t + \frac{f}{\epsilon} u^\perp + \frac{\beta}{\epsilon^2} \nabla
 \eta  + g(u) & = F, \\
\eta_t + \nabla \cdot \left( H u \right) & = 0.
\end{split}
\end{equation}

It also becomes useful to work in terms of the linearized momentum
$\widetilde{u} = H u$ rather than velocity. After making this substitution and dropping the
tildes, we obtain
\begin{equation}
\begin{split}
\frac{1}{H}u_t + \frac{f}{H\epsilon} u^\perp + \frac{\beta}{\epsilon^2} \nabla
\eta  + g(u) & = F, \\
\eta_t + \nabla \cdot  u & = 0.
\end{split}
\label{eq:thepde}
\end{equation}
A natural weak formulation of this equations is to seek $u \in \hdiv$
and $\eta \in L^2$ so that
\begin{equation}
\begin{split}
\left( \frac{1}{H}u_t , v \right) 
+ \frac{1}{\epsilon} \left( \frac{f}{H} u^\perp , v \right) 
- \frac{\beta}{\epsilon^2} \left( \eta ,
\nabla \cdot v \right) + \left( g(u) , v \right) & = 
\left( F , v \right)
, \quad \forall v \in \hdiv, \\
\left( \eta_t , w \right) + \left( \nabla \cdot u  , w \right)& = 0, \quad
\forall w \in L^2.
\end{split}
\label{eq:mixed}
\end{equation}

We now develop mixed discretizations with $V_h \subset \hdiv$ and
$W_h \subset L^2$.  Conditions on the spaces are the commuting projection
and divergence mapping $V_h$ onto $W_h$. We define $u_h:[0,T]\rightarrow V_h$
and $\eta_h :[0,T] \rightarrow  W_h$ as solutions of the discrete variational
problem
\begin{equation}
\begin{split}
\left( \frac{1}{H}u_{h,t} , v_h \right) 
+ \frac{1}{\epsilon} \left( \frac{f}{H} u_h^\perp , v_h \right) 
- \frac{\beta}{\epsilon^2} \left( \eta_h ,
\nabla \cdot v_h \right) + \left( g(u_h) , v_h \right) & =
\left( F , v_h \right)
, \\
\left( \eta_{h,t} , w_h \right) + \left( \nabla \cdot u_h  , w_h \right)& = 0.
\end{split}
\label{eq:discrete_mixed}
\end{equation}

Our analysis will proceed by working with an equivalent second-order
form.  While in the linear case~\cite{CoKi}, one readily 
obtains a second-order $\hdiv$ wave equation by differentiating the 
the first equation and using that $\nabla \cdot V_h = W_h$, this leads
to the somewhat awkward situation of differentiating through the nonlinearity.
A different approach allows us to avoid this unpleasantness.
Let satisfy the equation

\begin{equation}
  \frac{1}{H}\phi_{tt} + \frac{f}{H\epsilon} \phi_t^\perp -
  \frac{\beta}{\epsilon^2} \nabla \left( \nabla \cdot \phi \right) + g(\phi_t) = F
  \label{eq:secondorderpde}
\end{equation}
Then, we identify $u$ with $\phi_t$ and $\eta$ with $-\nabla \cdot
\phi$, and we see that solutions of~\eqref{eq:thepde}
and~\eqref{eq:secondorderpde} are in fact equivalent.  As an added
advantage over the technique in~\cite{CoKi}, the natural energy
functionals for the first- and second-order forms of the equation turn
out to coincide using this approach.

To analyze the semidiscrete setting, we need to adapt this observation
to the weak forms.  One may take the natural $\hdiv$ finite element
discretization of~\eqref{eq:secondorderpde}, seeking
$\phi_{h}: [0,T] \rightarrow V_h$ such that
\begin{equation}
  \left( \frac{1}{H}\phi_{h,tt}, v_h \right)
  + \left( \frac{f}{H\epsilon} \phi_{h,t}^\perp , v_h \right) +
  \frac{\beta}{\epsilon^2} \left( \nabla\cdot \phi_h , \nabla \cdot v_h
  \right) + \left( g(\phi_t) , v_h \right) = \left( F , v_h \right).
  \label{eq:secondorder_discrete}
\end{equation}
for all $v_h \in V_h$ for (almost) all $t \in [0,T]$.  Equivalently,
one could define $\phi_h$ to satisfy~\eqref{eq:secondorder_discrete}
and then note that standard properties of mixed finite element spaces
allow one to identify $u_h$ with $\phi_{h,t}$ and $\eta_h$ with
$\nabla \cdot \phi_h$ in~\eqref{eq:discrete_mixed}.

For the velocity space $V_h$, we will work with standard $\hdiv$ mixed
finite element spaces on triangular elements, such as Raviart-Thomas
(RT), Brezzi-Douglas-Marini (BDM), and Brezzi-Douglas-Fortin-Marini
(BDFM)~\cite{RavTho77a,brezzi1985two,brezzi1991mixed}.  We label the
lowest-order Raviart-Thomas space with index $k=1$, following the
ordering used in the finite element exterior
calculus~\cite{arnold2006finite}.  Similarly, the lowest-order
Brezzi-Douglas-Fortin-Marini and Brezzi-Douglas-Marini spaces
correspond to $k=1$ as well.  We will always take $W_h$ to consist of
piecewise polynomials of degree $k-1$, not constrained to be
continuous between cells.  We require the strong boundary condition
$u\cdot n = 0$ on all external boundaries.

Throughout, we shall let $\norm{\cdot}$ denote the standard $L^2$ norm.  We
will frequently work with weighted $L^2$ norms as well.  For a
positive-valued weight function $\kappa$, we define the weighted 
$L^2$ norm 
\begin{equation}
\weightednorm{v}{\kappa}^2
= \int_\Omega \kappa \left| v \right|^2 dx.
\end{equation}
If there exist positive constants $\kappa_*$ and $\kappa^*$ such that
$0 < \kappa_* \leq \kappa \leq \kappa^* < \infty$ 
almost everywhere, then the weighted norm is equivalent to
the standard $L^2$ norm by
\begin{equation}
\sqrt{\kappa_*} \norm{v} 
\leq \weightednorm{v}{\kappa} 
\leq \sqrt{\kappa^*} \norm{v}.
\end{equation}

A Cauchy-Schwarz inequality 
\begin{equation}
(\kappa v_1 , v_2) \leq 
\weightednorm{v_1}{\kappa}
\weightednorm{v_2}{\kappa}
\end{equation}
holds for the weighted inner product, and we can also incorporate
weights into Cauchy-Schwarz for the standard $L^2$ inner product by
\begin{equation}
(v_1,v_2) = 
(\sqrt{\kappa} v_1 , \frac{1}{\sqrt{\kappa}} v_2)
\leq \weightednorm{v_1}{\kappa} \weightednorm{v_2}{\frac{1}{\kappa}}.
\end{equation}

We refer the reader to references such as~\cite{brezzi1991mixed} for full
details about the particular definitions and properties of these
spaces, but here recall several facts essential for our analysis.  For
all velocity spaces $V_h$ we consider, the divergence maps $V_h$ onto $W_h$.
Also, the spaces of interest all have a projection, $\Pi : \hdiv
\rightarrow V_h$ that commutes with the $L^2$ projection $\pi$ into
$W_h$: 
\begin{equation}
\left( \nabla \cdot \Pi u , w_h \right)
= \left( \pi \nabla \cdot u , w_h \right)
\end{equation}
for all $w_h \in W_h$ and any $u \in \hdiv$.  
We have the error estimate
\begin{equation}
\norm{u - \Pi u} \leq C_{\Pi} h^{k+\sigma} \weightedseminorm{u}{k}
\label{eq:PiL2}
\end{equation}
when $u \in (H^{k+1})^2$.  Here, $\sigma = 1$ for the BDM spaces but
$\sigma =0$ for the RT or BDFM spaces. The projection also has an
error estimate for the divergence
\begin{equation}
\norm{ \nabla \cdot \left( u - \Pi u \right) } \leq C_\Pi h^{k} 
\weightedseminorm{\nabla \cdot u}{k}
\label{eq:PiDiv}
\end{equation}
for all the spaces of interest, whilst the pressure projection has the
error estimate 
\begin{equation}
\norm{ \eta - \pi \eta } \leq C_{\pi} h^k \weightedseminorm{\eta}{k}.
\label{eq:piL2}
\end{equation}
Here, $C_\Pi$ and $C_\pi$ are positive constants independent of $u$,
$\eta$, and $h$, although not necessarily of the shapes of the
elements in the mesh.

We will utilize a Helmholtz decomposition of $\hdiv$ under a weighted inner product.  For a very general treatment of such decompositions,
we refer the reader to~\cite{arnold2010finite}.  For each $u \in V$, there exist
unique vectors $u^D$ and $u^S$ such that $u = u^D + u^S$, $\nabla
\cdot u^S = 0$, and also $\left( \frac{1}{H} u^D , u^S \right) = 0$.
That is, $\hdiv$ is decomposed into the direct sum of the space of solenoidal
vectors, which we denote by
\begin{equation}
\mathcal{N} \left( \nabla \cdot \right)
= \left\{ u \in V : \nabla \cdot u = 0 \right\},
\end{equation}
 and its orthogonal complement under the $\left( \frac{1}{H}
  \cdot , \cdot \right)$ inner product, which we denote by
\begin{equation}
\mathcal{N} \left( \nabla \cdot \right)^\perp
= \left\{ u \in V : \left( \frac{1}{H} u , v \right) = 0, \ \forall v
  \in \mathcal{N} \left( \nabla \cdot \right) \right\}.
\end{equation}
Functions in $\mathcal{N} \left( \nabla \cdot \right)^\perp$
satisfy a generalized Poincar\'e-Friedrichs inequality, that there
exists some $C_P$ such that 
\begin{equation}
\weightednorm{u^D}{\frac{1}{H}}
\leq C_P \weightednorm{\nabla \cdot u^D}{\frac{1}{H}},
\label{eq:pf}
\end{equation}
or, via norm equivalence,
\begin{equation}
\weightednorm{u^D}{\frac{1}{H}}
\leq \frac{C_P}{\sqrt{H_*}} \norm{\nabla \cdot u^D}.
\end{equation}
Because our mixed spaces $V_h$ are contained in $H(\mathrm{div})$, the same decompositions can be applied, and the Poincar\'e-Friedrichs inequality holds with a constant no larger than $C_p$.

\section{Energy estimates}
\label{se:energy}
This section contains the major technical contributions of this paper.
We begin by considering the long-time energy boundedeness of the
system under our basic assumptions on $g$ in~\ref{ss:lts}.
Then, under more refined
assumptions, we study decay rates in~\ref{ss:decay}
and other continuous dependence results in~\ref{ss:cd}.

Throughout, we work with the energy functional
\begin{equation}
E(t) = \frac{1}{2} \weightednorm{ u_h}{\frac{1}{H}}^2 
+ \frac{\beta}{2\epsilon^2} \norm{\eta_h}^2
= \frac{1}{2} \weightednorm{\phi_{h,t}}{\frac{1}{H}}^2 
+ \frac{\beta}{2\epsilon^2} \norm{\nabla \cdot \phi_h}^2.
\label{eq:energy}
\end{equation}

It is easy to show that, absent forcing or damping ($F = g = 0$) (just
selecting $v_h = u_h$ and $w_h = \frac{\beta}{\epsilon^2}$
in~\eqref{eq:discrete_mixed} or $v_h = \phi_{h,t}$
in~\eqref{eq:secondorder_discrete}) that the the energy functional
is exactly conserved for all time.  With a nonzero damping 
satisfying~\eqref{eq:monot} and $F=0$, the energy cannot increase in time.
Just put $v_h = \phi_{h,t}$ in~\eqref{eq:secondorder_discrete} with
$F=0$ to
find that
\begin{equation}
\frac{d}{dt} E(t)
+ \left( g(\phi_{h,t}), \phi_{h,t} \right) = 0,
\end{equation}
and monotonicity gives that $\frac{d}{dt} E(t) \leq 0$.  However, this
is sufficient to show neither a rate at which $E(t) \rightarrow 0$ nor
that the damping is strong enough to give bounded energy when $F \neq
0$.  In the linear case, a more refined consideration actually gives
exponential energy decays as well as long-time stability, but such
results do not hold in the nonlinear case.

More generally, $v_h = \phi_{h,t}$ in~\eqref{eq:secondorder_discrete}
with nonzero forcing gives
\begin{equation}
  \label{eq:energy_relation}
\frac{d}{dt} E(t)
+ \left( g(\phi_{h,t}), \phi_{h,t} \right) = (F, \phi_{h,t}),
\end{equation}
and we refer to this as the \emph{energy relation} and will make
frequent use in our estimates.


\subsection{Long time stability}
\label{ss:lts}
We first address the question of long-time stability.  The
assumption of linear growth for large velocity will play a crucial
role here.

We begin with a simple lemma relating the damping term and some $L^2$
norms.
\begin{lemma}
  \label{lem1}
  Let $g$ satisfy~\eqref{eq:monot} and~\eqref{eq:lineargrowth}.  Then
  for all $v \in V_h$,
  \begin{equation}
    \| v \|^2 + \| g(v) \|^2
    \leq
    |\Omega| \left( 1 + g^*\right) + M \left( g(v), v \right),
  \end{equation}
  where
  \begin{equation}
    g^* \equiv \max_{|v| = 1} g(v).
  \end{equation}
\end{lemma}
\begin{proof}
  Let $v \in V_h$ be given.  We define
  \begin{equation}
    \begin{split}
      \Omega_0 \equiv & = \{ x \in \Omega : |v| < 1 \}, \\
      \Omega_1 \equiv & \Omega \setminus \Omega_1.
    \end{split}
  \end{equation}
  Then we calculate:
  \begin{equation}
    \begin{split}
      \| v \|^2 + \| g(v) \|^2 & = \int_\Omega |v|^2 + |g(v)|^2 \ dx
      \\
      & = \left( \int_{\Omega_0} + \int_{\Omega_1} \right)
        |v|^2 + |g(v)|^2 \ dx \\
        & \leq |\Omega_0| \left( 1 + g^* \right)
        + M \int_{\Omega_1} g(v) \cdot v \ dx.
    \end{split}
  \end{equation}
  The result follows by observing that $|\Omega_0| \leq |\Omega|$
  and that monotonicity allows us to bound the integral over
  $\Omega_1$ by that over all of $\Omega$.
\end{proof}

\begin{theorem}
  Suppose $g$ satisfies~\eqref{eq:monot} and~\eqref{eq:lineargrowth}
  and $F$ has a spatial $L^2$ norm uniformly bounded in time by
  $F^*$.  Then the energy of the solution $\phi_h$
  of~\eqref{eq:secondorder_discrete} remains uniformly bounded in time.
\end{theorem}
\begin{proof}
  We first put $v_h = \phi_h^D$ in~\eqref{eq:secondorder_discrete} to
  find
  \begin{multline} \label{eq:energybound1}
\frac{d}{dt}\left(\phi_{h,t},\phi_h^D \right)_{\frac{1}{H}} 
- \|\phi_{h,t}^D\|_{\frac{1}{H}}^2
 + \frac{1}{\epsilon}  \left(f\phi_{h,t}^\perp,\phi_h^D \right)_{\frac{1}{H}}
\\
+ \frac{\beta}{\epsilon^2} \|\nabla \cdot \phi_h\|^2 
+ \left(g(\phi_{h,t}),\phi_h^D\right) = (F,\phi_h^D).
  \end{multline}
Rearranging this and making estimates, we have
\begin{multline}
  \label{eq:energybound2}
  \frac{d}{dt}\left(\phi_{h,t},\phi_h^D \right)_{\frac{1}{H}}
  + \frac{\beta}{\epsilon^2} \|\nabla \cdot \phi_h\|^2
  + \|\phi_{h,t} \|_{\frac{1}{H}}^2
  \leq
  2 \|\phi_{h,t} \|_{\frac{1}{H}}^2
  + \frac{f^* C_P}{\epsilon} \|\phi_{h,t} \|_{\frac{1}{H}}
  \|\nabla \cdot \phi_h\| \\
  + \frac{C_P \sqrt{H^*}}{\sqrt{H_*}}
  \| g(\phi_{h,t}) \|
  \|\nabla \cdot \phi_h\|
  + \frac{C_P \sqrt{H^*}}{\sqrt{H_*}} \| F \|
    \|\nabla \cdot \phi_h\|.
\end{multline}
Then, Young's inequality on each product in the right-hand side (using
the same delta in the second and third products) gives
\begin{multline}
  \label{eq:energybound3}
  \frac{d}{dt}\left(\phi_{h,t},\phi_h^D \right)_{\frac{1}{H}}
  + \frac{\beta}{\epsilon^2} \|\nabla \cdot \phi_h\|^2
  + \|\phi_{h,t} \|_{\frac{1}{H}}^2 \\
  \leq
  \left( 2 + \frac{f^*C_P}{2\epsilon\delta_1} \right) \|\phi_{h,t} \|_{\frac{1}{H}}^2
  + \left( \frac{f^*C_P}{2} \delta_1 + \frac{C_P\sqrt{H^*}}{\sqrt{H_*}}\delta_2
  \right)  \| \nabla \cdot \phi_h \|^2 \\
  + \frac{C_P\sqrt{H^*}}{2\sqrt{H_*}\delta_2} \|F\|^2
  + \frac{C_P\sqrt{H^*}}{2\sqrt{H_*}\delta_2} \|g(\phi_{h,t})\|^2.
\end{multline}
Our goal is to hide the divergence on the left-hand side and then use
Lemma~\ref{lem1} and the energy relation~\eqref{eq:energy_relation} to handle the $L^2$ norms
of $\phi_{h,t}$ and $g(\phi_{h,t})$.  To this end, we put
\[\delta_1 = \frac{\beta}{3f^*C_P\epsilon}, \ \ \
\delta_2 = \frac{\beta\sqrt{H_*}}{3\epsilon^2 C_P \sqrt{H_*}}\]
so that
\begin{multline}
  \label{eq:energybound4}
  \frac{d}{dt}\left(\phi_{h,t},\phi_h^D \right)_{\frac{1}{H}}
  + \frac{\beta}{2\epsilon^2} \|\nabla \cdot \phi_h\|^2
  + \|\phi_{h,t} \|_{\frac{1}{H}}^2 \\
  \leq A_3 \left( \| \phi_{h,t} \|^2 + \| g(\phi_{h,t})  \|^2 \right) + A_2 \| F \|^2,
\end{multline}
where
\[
\begin{split}
  A_1 & \equiv \left( 2 + \frac{3(f^*C_P)^2}{2\beta}
  \right)\frac{1}{\sqrt{H_*}}, \\
  A_2 & \equiv \frac{3(\epsilon C_P)^2 H^*}{2\beta H_*}, \\
  A_3 & \equiv \max\{A_1, A_2\}.
\end{split}
\]
Then, Lemma~\ref{lem1} gives
\begin{multline}
  \label{eq:energybound5}
  \frac{d}{dt}\left(\phi_{h,t},\phi_h^D \right)_{\frac{1}{H}}
  + \frac{\beta}{2\epsilon^2} \|\nabla \cdot \phi_h\|^2
  + \|\phi_{h,t} \|_{\frac{1}{H}}^2
  \\ \leq A_4 + \tilde{A}_3 M \left( g(\phi_{h,t}), \phi_{h,t} \right)
  + A_2 \| F \|^2,
\end{multline}
where
\begin{equation}
A_4 \equiv A_3 |\Omega| \left( 1 + g^* \right),
\end{equation}
and $\tilde{A}_3 \geq A_3$ will be fixed later.  Applying the energy
relation leads to
\begin{multline}
  \frac{d}{dt}\left[ \tilde{A}_3 M E(t) + \left(\phi_{h,t},\phi_h^D
    \right)_{\frac{1}{H}} \right]
  + \frac{\beta}{2\epsilon^2} \|\nabla \cdot \phi_h\|^2
  + \|\phi_{h,t} \|_{\frac{1}{H}}^2
\\ \leq A_4 + \tilde{A}_3M\left(F, \phi_{h,t}\right) + A_2 \| F \|^2.
\end{multline}
Now, a weighted Young's inequality and norm equivalences allow us to
write
\begin{multline}
    \frac{d}{dt}\left[ \tilde{A}_3 M E(t) + \left(\phi_{h,t},\phi_h^D
      \right)_{\frac{1}{H}} \right] + E(t)
    \leq A_4 + \left[ \frac{(\tilde{A}_3M)^2 H^*}{2}+A_2\right] \|F \|^2.
\end{multline}
We divide through by $A_5 \equiv \tilde{A}_3M$ and define
\begin{equation}
  A(t) = E(t) + \frac{1}{A_5} \left( \phi_{h,t}, \phi_h^D
  \right)_{\frac{1}{H}}
\end{equation}
so that
\begin{equation}
  \label{eq:almostthere}
  \frac{d}{dt} A(t) + \frac{1}{A_5} E(t) \leq A_6 + A_7 \|F\|^2,
\end{equation}
where
\begin{equation}
  A_6 \equiv \frac{A_4}{A_5}, \ \
  A_7 \equiv \frac{(\tilde{A}_3M)^2H^* + 2A_2}{A_5}.
\end{equation}
At this point, we have an ordinary differential inequality, and we are
able to choose $\tilde{A}_3$ in order to guarantee an equivalence
between $A(t)$ and $E(t)$.  

Since we observe that
\begin{equation}
\left|\left(\phi_{h,t},\phi_h^D \right)_{\frac{1}{H}}\right| \leq \frac{C_P}{\sqrt{H_*}}\|\phi_{h,t}\|_{\frac{1}{H}}\|\nabla \cdot \phi\| \leq \frac{C_P \epsilon}{\sqrt{\beta H_*}}E(t),
\end{equation}
we set
\begin{equation}
\tilde A_3 = \max\left\{A_3,\frac{2C_P \epsilon}{M\sqrt{\beta H_*}} \right\},
\end{equation}
which readily gives that
\begin{equation}
\label{eq:equivalence}
\frac{1}{2}E(t) \leq A(t) \leq \frac{3}{2}E(t).
\end{equation}

At this point, we use this equivalence to
convert~\eqref{eq:almostthere} to an ordinary differential inequality for
$A(t)$ to determine
\begin{equation}
  A(t) \leq e^{-\frac{2t}{3A_5}} A(0) + \int_0^t
  e^{\frac{2(s-t)}{3A_5}} \left( A_6 + A_7 \| F \|^2 \right) ds
\end{equation}
and hence
\begin{equation}
  E(t) \leq e^{-\frac{2t}{3A_5}} E(0) + \int_0^t
  e^{\frac{2(s-t)}{3A_5}} \left( A_6 + A_7 \| F \|^2 \right) ds.
\end{equation}
Finally, computing the integral and using that $\| F\|^2 \leq F^*$
gives
\begin{equation}
E(t) \leq e^{-\frac{2t}{3A_5}} E(0) +
\frac{3A_5}{2} \left( 1-e^{-\frac{2t}{3A_5}} \right) \left( A_6 + A_7
F^* \right)
\end{equation}
for all time.
\end{proof}
This result demonstrates that our model remains stable for all times.
The bound eventually becomes independent of the initial energy,
although this does not yet prove the existence of an attracting
solution.
Also, note that, as $F^* \rightarrow 0$, we only obtain an
$\mathcal{O}(1)$ bound on $E(t)$.  Looking ahead to error estimation,
this result could be used to show that error remains uniformly bounded
in time, but cannot be used to establish convergence rates as $h
\rightarrow 0$. 

\subsection{Decay rates}
\label{ss:decay}
Now, we turn back to the question of $F = 0$ and determine that any
initial energy must decay toward 0 at a rate that is determined by
features of the nonlinearity.  To establish this will require
stronger assumptions on the nonlinearity $g$.  However, we will
actually prove a result on differences of solutions subject to
identifical forcing but different initial data.  This will establish
decay rates and rates of convergence to a global attracting solution.

In particular, we now require that $g$ is a continuous function (of
both variables) and that
\begin{itemize}
  \item Mononoticity:
    \begin{equation}
      \label{eq:monotone-strict}
      (g(v) - g(w))\cdot(v-w) > 0
    \end{equation}
    for all $v \neq w$, uniformly in the implicit $x$-dependence.
  \item
    Linear growth also holds on differences.  That is, for some
    $M > 0$ 
    \begin{equation}
      \label{eq:linearatinfinity}
      |v-w|^2, |g(v)-g(w)|^2 \leq Mg(v-w)\cdot (v-w)
    \end{equation}
    for all $|v|, |w| \geq 1$, again uniformly in $x$.
\end{itemize}

\emph{Remark:} If one were interested only in decay rates for a single solution given $F = 0$, then \eqref{eq:monotone-strict} could be reduced to $g(v) \cdot v > 0$ for all $v \neq 0$, and \eqref{eq:linearatinfinity} could be analogously reduced.

The technique used in this section was first developed by Lasiecka and Tataru in \cite{lasiecka1993uniform}, where the main purpose was to prove the existence of uniform decay rates for the wave equation with nonlinear boundary damping.
See also \cite{cavalcanti2007well} for an extension of the method as well as an overview of the relevant PDE literature.
Our main interest in \cite{lasiecka1993uniform} is that it provides an algorithm which takes the profile of \emph{any} monotone damping function $g$ and produces an explicit uniform decay rate for the energy.
While most natural examples of $g$ have the structure of a power law, the existence of a decay rate is in fact generic; it depends only on the fact that $g$ is monotone and sufficiently dampens high velocities.

\subsubsection{Some lemmas}

Our results will depend on a few technical lemmas.
The first lemma appears in \cite{lasiecka1993uniform} as a brief remark, but there it is applied only to the case where $g$ is a scalar monotone function.
Here we generalize to the case where $g$ is a vector field.
\begin{lemma} \label{lem:auxiliaryfunction2}
	Let $g = g(x,v)$ be a continuous function on $\overline{\Omega} \times \mathbb{R}^d$, where $\Omega$ is a bounded domain, satisfying  \eqref{eq:monotone-strict} and \eqref{eq:linearatinfinity}.
	Then there exists an increasing, concave function $J:[0,\infty) \to [0,\infty)$ such that $J(0) = 0$ and
	\begin{equation} \label{eq:J2}
	|v-w|^2 + |g(v)-g(w)|^2 \leq J((v-w) \cdot (g(v)-g(w))) \ \ \ \forall |w|,|v| \leq 1, \ \forall x \in \overline{\Omega}.
	\end{equation}
\end{lemma}

\begin{proof}
	Let $B_1 = \{v \in \mathbb{R}^d : |v| \leq 1\}$ and let $\partial B_1$ be its boundary.
	For $v \in B_1, e \in \partial B_1$ and $x \in \overline{\Omega}$, set
	$$
	h_{v,e}(s) = se \cdot (g(v+se)-g(v)), \ \ j_{v,e}(s) = s^2 + \max\{|g(v+te)-g(v)|^2 : 0 \leq t \leq s\}.
	$$
	Note that both functions are strictly increasing in $s$; $j_{e,x}$ is the sum of two increasing functions, one of them strictly increasing, while in the case of $h_{v,e,x}(s)$, we use \eqref{eq:monotone-strict} to check:
	\begin{align*}
	s > t \ \Rightarrow h_{v,e}(s) &- h_{v,e}(t) 
	\\
	&= se \cdot (g(v+se)-g(v)) - te \cdot (g(v+te)-g(v))\\
	&> (s-t)e \cdot (g(v+te)-g(v)) \geq 0.
	\end{align*}
	Moreover by \eqref{eq:linearatinfinity} we have that $h_{v,e}(s) \to \infty$ as $s \to \infty$.
	Let $h_{v,e}^{-1}:[0,\infty) \to [0,\infty)$ be the inverse function of $h_{v,e}(\cdot)$.
	Our goal is to show that 
	$$
	j(t) := \max_{(v,e) \in B_1 \times \partial B_1 } j_{v,e}(h_{e}^{-1}(t))
	$$
	exists (that is, it is finite for all $t$).
	To do this, it is sufficient to see that $j_{v,e}(s)$ and
        $h_{v,e}^{-1}(t)$ are both continuous in the stripe $(v,e)$
        (uniformly in $x$).
	The continuity of $(v,e) \mapsto j_{v,e}(s)$ follows in a straightforward manner from the uniform continuity of $g$ on compact sets.
	Likewise, $h_{v,e}(s)$ is continuous in $(v,e)$.
	To see that $h_{v,e}^{-1}(t)$ is continuous in $(v,e)$, we assume to the contrary that there exists some sequence $(v_n,e_n) \in B_1 \times \partial B_1 $ such that $(v_n,e_n) \to (v,e)$ while $|h_{v_n,e_n}^{-1}(t) - h_{v,e}^{-1}(t)| \geq \epsilon$.
	Let $s_n = h_{v_n,e_n}^{-1}(t)$ and $s = h_{e}^{-1}(t)$.
	There are two cases:
	\begin{enumerate}
		\item[1:] $s_n \geq s + \epsilon$ (up to a subsequence).
		Since $h_{e_n}$ and $h_{e}$ are strictly increasing, it follows that $h_{e_n}(s_n) \geq h_{e_n}(s+\epsilon) \to h_{e}(s+\epsilon) > h_{e}(s)$.
		But this implies $t > t$, a contradiction.
		
		\item[2:] $s_n \leq s - \epsilon$ (up to a subsequence).
		We have $h_{e_n}(s_n) \leq h_{e_n}(s-\epsilon) \to h_{e}(s-\epsilon) < h_{e}(s)$, so $t < t$, a contradiction.
	\end{enumerate}
	We now see that $h_{v,e}^{-1}(t)$ is continuous in $(v,e)$ for every $t \geq 0$.
	
	To complete the proof, observe that $j(t)$ is well-defined and finite for all $t \geq 0$, that $j(0) = 0$, and $j$ is increasing.
	Set
	\begin{multline*}
	t_1 := \max\{(v-w) \cdot (g(v)-g(w)) : v,w \in B_1\} 
	\\
	= \max\{h_{v,e}(s) : v,v+se \in B_1\}.
	\end{multline*}
	Finally, let $J$ be the concave envelope of $j$ restricted to $[0,t_1]$ (and constant on $[t_1,\infty)$).
	Then $J$ satisfies all the desired properties.
\end{proof}

The function $J$ derived in Lemma \ref{lem:auxiliaryfunction2} determines the decay rates via an ordinary differential equation \eqref{eq:ode-final}.
Loosely speaking, it determines how much the damping is able to ``coerce" the energy.
We note that \eqref{eq:J2} only applies to vectors in the unit ball.
On the other hand, for vectors outside the unit ball, we can use the structure assumed in \eqref{eq:linearatinfinity}.
For the case when one vector is inside the unit ball while the other is outside, we will appeal to this elementary lemma, which is a corollary of \eqref{eq:monotone-strict}:

\begin{lemma} \label{lem:elementary}
	Given the above assumptions on $g$, then if $|v| \geq 1$ and $|w| < 1$ (or vice versa), we have
	\begin{multline} \label{eq:elementary}
	|v-w|^2,|g(v) - g(w)|^2
	\\ \leq 2M((g(v)-g(w))\cdot (v-w)) + 2J((g(v)-g(w))\cdot (v-w))
	\end{multline}
	for all $x \in \Omega$.
	\end{lemma}

\begin{proof}
	Let $|v| \geq 1$ and $|w| < 1$.
	Set $v_\lambda = \lambda w + (1-\lambda)v$ and fix $\lambda \in (0,1]$ such that $|v_\lambda| = 1$.
	Using the identities $v-v_\lambda = \lambda(v-w)$ and $v_\lambda - w = (1-\lambda)(v-w)$, the fact that $g(x,\cdot)$ is monotone and satisfies \eqref{eq:linearatinfinity}, and Lemma \ref{lem:auxiliaryfunction2}, we get
	\begin{align*}
	|v-w|^2 &\leq 2|v-v_\lambda|^2 + 2|v_\lambda-w|^2\\
	&\leq 2M(g(v)-g(v_\lambda))\cdot (v-v_\lambda) + 2J((g(v_\lambda)-g(w))\cdot (v_\lambda-w))
	\\
	&\leq 2M(g(v)-g(w))\cdot (v-w) + 2J((g(v)-g(w))\cdot (v-w))
	\end{align*}
	The second part of \eqref{eq:elementary} is similar, and we omit the details.
	\end{proof}

\subsubsection{Derivation of decay rates}

Let $\phi_{1,h},\phi_{2,h}$ be two solutions of \eqref{eq:secondorder_discrete} with different initial data.
Set $\phi_h = \phi_{1,h}-\phi_{2,h}$.
Then $\phi_h$ satisfies
\begin{multline}
\left( \frac{1}{H}\phi_{h,tt}, v_h \right)
+ \left( \frac{f}{H\epsilon} \phi_{h,t}^\perp , v_h \right) +
\frac{\beta}{\epsilon^2} \left( \nabla\cdot \phi_h , \nabla \cdot v_h
\right) 
\\
+ \left( g(\phi_{1,h,t})-g(\phi_{2,h,t}) , v_h \right) = 0.
\label{eq:secondorder_discrete_differences}
\end{multline}
for all $v_h \in V_h$ for (almost) all $t \in [0,T]$.
We will once again define
\begin{equation}
E(t) 
= \frac{1}{2} \weightednorm{\phi_{h,t}}{\frac{1}{H}}^2 
+ \frac{\beta}{2\epsilon^2} \norm{\nabla \cdot \phi_h}^2,
\label{eq:energy-differences}
\end{equation}
and again we have the energy identity
\begin{equation}
  \label{eq:enid}
\frac{d}{dt} E(t)
+ \left( g(\phi_{1,h,t})-g(\phi_{2,h,t}), \phi_{h,t} \right) = 0.
\end{equation}

Our main theorem of this section bounds the energy by the
solution of an ordinary differential equation, where this equation is
obtained in terms of the concave function $J$ given above.  For
particular choices of $g$, one may explicitly compute $J$ and hence
the solution of the ODE.  Examples will follow after the theorem.

\begin{theorem}
  Let $E(t)$ be defined in~\eqref{eq:enid}. 
  Then for all $t \geq T$, the energy $E(t)$ satisfies
  \begin{equation}
    E(t) \leq S\left(\frac{t}{T} - 1\right),
  \end{equation}
  where $S$ is the solution to
  \begin{equation}
    S^\prime(t) + |\Sigma| J^{-1} \left( \frac{S(t)}{D_J} \right) = 0,
    \ \ S(0) = E(0),
  \end{equation}
  and where
  \begin{equation}
    \begin{split}
      T & := 2\frac{C_P\sqrt{\beta}}{\epsilon \sqrt{H_*}}, \\
      |\Sigma| & := |\Omega| T. \\
      D_1 & := 2M\left(\frac{3}{2}+\frac{f^* C_P^2}{\beta H_*} \right)\frac{1}{H_*}
 + 2M\frac{C_P^2 H^* \epsilon^2}{\beta H_*}, \\
 D_2 & := 2\left(\frac{3}{2}+\frac{f^* C_P^2}{\beta H_*}
 \right)\frac{1}{H_*} + 2\frac{C_P^2 H^* \epsilon^2}{\beta H_*}, \\
 \tilde D_1 & := \frac{2C_P\sqrt{\beta}}{\epsilon \sqrt{H_*}}+ D_1
= \frac{2C_P\sqrt{\beta}}{\epsilon \sqrt{H_*}}+ 2M\left(\frac{3}{2}+\frac{f^* C_P^2}{\beta H_*} \right)\frac{1}{H_*}
+ 2M\frac{C_P^2 H^* \epsilon^2}{\beta H_*} \\
 D_J & := \left(1+ \tilde D_1\right)\frac{E(0)}{J\left(\frac{E(0)}{|\Sigma|}\right)}
+ D_2|\Sigma|.
    \end{split}
  \end{equation}
\end{theorem}

\begin{proof}

\textit{Step 1.}
Take $v_h = \phi_h^D$ in \eqref{eq:secondorder_discrete_differences} and integrate in time.  Integration by parts gives
\begin{multline} \label{eq:energyestimate1}
\left. \left(\phi_{h,t},\phi_h^D \right)_{\frac{1}{H}} \right|_0^T 
- \int_0^T \|\phi_{h,t}^D\|_{\frac{1}{H}}^2 dt + \frac{1}{\epsilon} \int_0^T \left(f\phi_{h,t}^\perp,\phi_h^D \right)_{\frac{1}{H}}dt 
\\
+ \frac{\beta}{\epsilon^2} \int_0^T \|\nabla \cdot \phi_h\|^2 dt 
+ \int_0^T \left(g(\phi_{1,h,t})-g(\phi_{2,h,t}),\phi_h^D\right) dt = 0.
\end{multline}
Here and in the following we use that $\nabla \cdot \phi_h =\nabla \cdot \phi_h^D$.
Using the Cauchy-Schwarz inequality and \eqref{eq:pf}, quation \eqref{eq:energyestimate1} becomes
\begin{multline} \label{eq:energyestimate2}
\frac{\beta}{\epsilon^2} \int_0^T \|\nabla \cdot \phi_h\|^2 dt 
\leq
\frac{C_P}{\sqrt{H_*}}\|\phi_{h,t}(T)\|_{\frac{1}{H}}\|\nabla \cdot \phi_h(T)\|
+ \frac{\epsilon^2 C_P}{\beta \sqrt{H_*}}\|\phi_{h,t}(0)\|_{\frac{1}{H}}\|\nabla \cdot \phi_h(0)\| \\
+ \int_0^T \|\phi_{h,t}^D\|_{\frac{1}{H}}^2 dt 
 + \frac{f^* C_P}{\epsilon\sqrt{H_*}} \int_0^T \|f\phi_{h,t}^\perp\|_{\frac{1}{H}}\|\nabla \cdot \phi_h\|dt \\
+ \frac{C_P\sqrt{H^*}}{\sqrt{H_*}}\int_0^T \|g(\phi_{1,h,t})-g(\phi_{2,h,t})\|\|\nabla \cdot \phi_h\| dt.
\end{multline}
We handle the terms at time $T$ and $0$ by the weighted inequality
$ab \leq \frac{a^2}{2\delta} + \frac{b^2\delta}{2}$ with $\delta= \frac{\epsilon}{\sqrt{\beta}}$.  Then, we pull out $f^*$ from $\| f \phi_{h,t}^\perp \|$ and use that $\| \phi_{h,t}^D \| \leq \| \phi_{h,t} \|$ and that $\cdot^\perp$ is an isometry to obtain
\begin{multline}
  \label{eq:energyestimate2.5}
  \frac{\beta}{\epsilon^2} \int_0^T \|\nabla \cdot \phi_h\|^2 dt
  \leq \frac{C_P\sqrt{\beta}}{\epsilon \sqrt{H_*}} \left[ E(T) + E(0) \right]
  + \int_0^T \|\phi_{h,t}\|_{\frac{1}{H}}^2 dt  \\
 + \frac{f^* C_P}{\epsilon\sqrt{H_*}} \int_0^T \|\phi_{h,t}\|_{\frac{1}{H}}\|\nabla \cdot \phi_h\|dt
+ \frac{C_P\sqrt{H^*}}{\sqrt{H_*}}\int_0^T \|g(\phi_{1,h,t})-g(\phi_{2,h,t})\|\|\nabla \cdot \phi_h\| dt.
\end{multline}
Next, we handle the terms under the integrals with the same weighted inequality.  In the first case we use $\delta = \frac{\beta\sqrt{H_*}}{2\epsilon f^* C_P}$ and in the second we use $\delta = \frac{\beta\sqrt{H_*}}{2\epsilon^2 C_P \sqrt{H^*}}$.  Then, collecting terms and using that $E(T) \leq E(0)$, we have
\begin{multline} \label{eq:energyestimate3}
  \frac{\beta}{\epsilon^2} \int_0^T \| \nabla \cdot \phi_h \|^2 dt
  \leq \frac{2C_P\sqrt{\beta}}{\epsilon \sqrt{H_*}} E(0) \\
  + \left(1+\frac{f^* C_P^2}{\beta H_*} \right)
  \int_0^T \weightednorm{\phi_{h,t}}{\frac{1}{H}}^2 dt
  + \frac{C_P^2 H^* \epsilon^2}{\beta H_*}
  \int_0^T \norm{g(\phi_{h,t})}^2 dt
\end{multline}

So then, it follows that
\begin{multline}\label{eq:energyestimate4}
\int_0^T E(t)dt = \int_0^T \left(\frac{1}{2} \weightednorm{\phi_{h,t}}{\frac{1}{H}}^2 
+ \frac{\beta}{2\epsilon^2} \norm{\nabla \cdot \phi_h}^2  \right)dt \\
  \leq \frac{2C_P\sqrt{\beta}}{\epsilon \sqrt{H_*}} E(0) \\
  + \left(\frac{3}{2}+\frac{f^* C_P^2}{\beta H_*} \right)
  \int_0^T \weightednorm{\phi_{h,t}}{\frac{1}{H}}^2 dt
  + \frac{C_P^2 H^* \epsilon^2}{\beta H_*}
  \int_0^T \norm{g(\phi_{h,t})}^2 dt
\end{multline}


\textit{Step 2.}
Set $\Sigma := \Omega \times (0,T)$.
Rewrite \eqref{eq:energyestimate4} as
\begin{multline}
\label{eq:energyestimate5}
\int_0^T E(t)dt \leq \frac{2C_P\sqrt{\beta}}{\epsilon \sqrt{H_*}} E(0) \\
  + \left(\frac{3}{2}+\frac{f^* C_P^2}{\beta H_*} \right)\frac{1}{H_*}
  \int_\Sigma |\phi_{h,t}|^2 \ dxdt
  + \frac{C_P^2 H^* \epsilon^2}{\beta H_*}
  \int_\Sigma |g(\phi_{h,t})|^2 \ dx dt.
\end{multline}
Define
$$
\Sigma_0 = \{(x,t) \in \Sigma : |\phi_{1,h,t}(x,t)|,|\phi_{2,h,t}(x,t)| \leq 1\}, \ \Sigma_1 = \Sigma \setminus \Sigma_0.
$$
We can break down $\Sigma_1$ further into
$$
\Sigma_{1,1} = \{(x,t) \in \Sigma : |\phi_{1,h,t}(x,t)|,|\phi_{2,h,t}(x,t)| \geq 1\}, \ \Sigma_{1,0} = \Sigma_1 \setminus \Sigma_{1,1}.
$$
Then we find, using Assumption \eqref{eq:linearatinfinity} and Lemmas \ref{lem:auxiliaryfunction2} and \ref{lem:elementary}, that
\begin{multline} \label{eq:differencesJ1}
\int_\Sigma |\phi_{h,t}|^2 \ dxdt
= \left(\int_{\Sigma_0} + \int_{\Sigma_{1,0}} + \int_{\Sigma_{1,1}}\right) |\phi_{h,t}|^2 \ dxdt
\\
\leq \int_{\Sigma_0} J(\phi_{h,t} \cdot (g(\phi_{1,h,t})-g(\phi_{2,h,t}))) \ dxdt
\\
+  \int_{\Sigma_{1,0}} \left\{2J(\phi_{h,t} \cdot (g(\phi_{1,h,t})-g(\phi_{2,h,t}))) + 2M\phi_{h,t} \cdot (g(\phi_{1,h,t})-g(\phi_{2,h,t}))\right\} \ dxdt
\\
+ \int_{\Sigma_{1,1}} M\phi_{h,t} \cdot (g(\phi_{1,h,t})-g(\phi_{2,h,t})) \ dxdt
\\
\leq 2M \int_\Sigma \phi_{h,t} \cdot (g(\phi_{1,h,t})-g(\phi_{2,h,t})) \ dxdt
+ 2\int_\Sigma J\left(\phi_{h,t} \cdot (g(\phi_{1,h,t})-g(\phi_{2,h,t}))\right) \ dxdt.
\end{multline}
In the same way,
\begin{multline} \label{eq:differencesJ2}
\int_\Sigma |g(\phi_{1,h,t})-g(\phi_{2,h,t})|^2 \ dxdt
\leq 2M \int_\Sigma \phi_{h,t} \cdot (g(\phi_{1,h,t})-g(\phi_{2,h,t})) \ dxdt
\\
+ 2\int_\Sigma J\left(\phi_{h,t} \cdot (g(\phi_{1,h,t})-g(\phi_{2,h,t}))\right) \ dxdt.
\end{multline}
Inserting \eqref{eq:differencesJ1} and \eqref{eq:differencesJ2} into \eqref{eq:energyestimate5} we get
\begin{multline}\label{eq:difference-estimate2}
\int_0^T E(t)dt 
\leq \frac{2C_P\sqrt{\beta}}{\epsilon \sqrt{H_*}} E(0) 
+ D_1 \int_\Sigma \phi_{h,t} \cdot (g(\phi_{1,h,t})-g(\phi_{2,h,t})) \ dxdt
\\
+ D_2\int_\Sigma J\left(\phi_{h,t} \cdot (g(\phi_{1,h,t})-g(\phi_{2,h,t}))\right) \ dxdt
\end{multline}
where
\begin{equation} \label{eq:D1D2}
D_1 = 2M\left(\frac{3}{2}+\frac{f^* C_P^2}{\beta H_*} \right)\frac{1}{H_*}
 + 2M\frac{C_P^2 H^* \epsilon^2}{\beta H_*}
,
\ \
D_2 = 2\left(\frac{3}{2}+\frac{f^* C_P^2}{\beta H_*} \right)\frac{1}{H_*} + 2\frac{C_P^2 H^* \epsilon^2}{\beta H_*}.
\end{equation}
Recall Jensen's inequality: since $J$ is concave and nonnegative,
\begin{multline}
\int_\Sigma J\left(\phi_{h,t} \cdot (g(\phi_{1,h,t})-g(\phi_{2,h,t}))\right) \ dxdt
\\
\leq |\Sigma|J\left(\frac{1}{|\Sigma|}\int_\Sigma\phi_{h,t} \cdot (g(\phi_{1,h,t})-g(\phi_{2,h,t})) \ dxdt\right)
\end{multline}
Then since $\int_\Sigma\phi_{h,t} \cdot (g(\phi_{1,h,t})-g(\phi_{2,h,t})) \ dxdt = E(0) - E(T)$, we can deduce from \eqref{eq:difference-estimate2} that
\begin{equation}\label{eq:difference-estimate3}
\int_0^T E(t)dt 
\leq \frac{2C_P\sqrt{\beta}}{\epsilon \sqrt{H_*}} E(0) 
+ D_1(E(0)-E(T))
+ D_2|\Sigma|J\left(\frac{E(0)-E(T)}{|\Sigma|}\right).
\end{equation}
Since $E(t)$ is monotone decreasing, \eqref{eq:difference-estimate3} yields
\begin{equation}\label{eq:difference-estimate4}
E(T)
\leq \tilde D_1( E(0) -E(T))
+ D_2|\Sigma|J\left(\frac{E(0)-E(T)}{|\Sigma|}\right)
\end{equation}
where

We define a strictly increasing function $p(s)$ by defining its inverse:
\begin{equation}
p^{-1}(s) = \tilde D_1 s
+ D_2|\Sigma|J\left(\frac{s}{|\Sigma|}\right).
\end{equation}
It follows that
\begin{equation}
E(T) + p(E(T)) \leq E(0).
\end{equation}
By repeating the same argument on any time interval, we get
\begin{equation}
E((n+1)T) + p(E((n+1)T)) \leq E(nT), \ \ n = 1,2,3,\ldots
\end{equation}
We now appeal to Lemma 3.3 and the argument that follows in (Lasiecka-Tataru 1993) to assert 
\begin{equation}
E(t) \leq S\left(\frac{t}{T}-1\right) \ \forall t \geq T,
\end{equation}
where $S$ solves the ordinary differential equation
\begin{equation} \label{eq:ODE-q}
S'(t) + q(S(t)) = 0, \ S(0) = E(0)
\end{equation}
and $q$ is any increasing function such that $q \leq I - (I+p)^{-1} = (I+p^{-1})^{-1}$.


\textit{Step 3.}
To find an appropriate $q$, we estimate $(I+p^{-1})^{-1}$ or, equivalently, $I+p^{-1}$, which is given by
\begin{equation}
(I+p^{-1})(s) = (1+\tilde D_1) s
+ D_2|\Sigma|J\left(\frac{s}{|\Sigma|}\right).
\end{equation}
Note that since $S(t)$ will always be positive and bounded above by $E(0)$, it suffices to restrict our attention only to the interval $[0,E(0)]$.
Since $J$ is concave and $J(0) = 0$, we can write
\begin{equation}
J\left(\frac{s}{|\Sigma|}\right) \geq \frac{s}{E(0)}J\left(\frac{E(0)}{|\Sigma|}\right)
\ \ \
\forall s \in [0,E(0)].
\end{equation}
Therefore,
\begin{equation} \label{eq:inverse-estimate}
  (I+p^{-1})(s) \leq D_J J\left(\frac{s}{|\Sigma|}\right)
\end{equation}
Inverting \eqref{eq:inverse-estimate} we see that an appropriate $q$ is given by
\begin{equation}
q(s) := |\Sigma|J^{-1}\left(\frac{s}{D_J}\right),
\end{equation}
i.e.~$S$ can be taken in the solution of the ODE
\begin{equation} \label{eq:ode-final}
S'(t) + |\Sigma|J^{-1}\left(\frac{S(t)}{D_J}\right) = 0, \ S(0) = E(0).
\end{equation}

\end{proof}

\textbf{Examples.}
Let $p > 1$ and set
\begin{equation}
\label{eq:pgrowth}
g(x,v) = g(v) = \left\{
\begin{array}{cc}
|v|^{p-2}v & \text{if} \ |v| \leq 1\\
v & \text{if} \ |v| \geq 1
\end{array}\right.
\end{equation}
When $p > 2$ we refer to this as \textit{superlinear growth} while $p < 2$ is called \textit{sublinear growth}.

\textit{Superlinear growth:} If $p > 2$, we have
$$
||v|^{p-2}v-|w|^{p-2}w| \leq |v-w| \ \forall v,w \in B_1
$$
and so \eqref{eq:J2} can be replaced by
$$
2|v-w|^2 \leq J((v-w) \cdot (|v|^{p-2}v-|w|^{p-2}w)).
$$
Now on the other hand, we have
$$
(v-w) \cdot (|v|^{p-2}v-|w|^{p-2}w) \geq \frac{1}{2^{p-2}}|v-w|^p \ \forall v,w.
$$
This can be proved by vector calculus.
Thus it suffices to choose $J(s) = 2^{3-4/p}s^{2/p}$.
In this case the ODE \eqref{eq:ode-final} becomes
\begin{equation} \label{eq:ode-superlinear}
S'(t) + \frac{2^{2-3p/2}|\Sigma|}{D_J^{p/2}}S(t)^{p/2} = 0, \ S(0) = E(0).
\end{equation}

To give the decay rates for these superlinear power laws, separation of variables on the ODE $S^\prime + \gamma S = 0$ leads to the solution
\[
S(t) = \left[ \left(\frac{p}{2} - 1 \right)
  \left( \gamma t - C \right) \right]^{\frac{1}{1-p/2}},
\]
where $C$ is an additive constant set to make $S(0) = E(0)$.  In this case, we can plug in $p=3$, the quadratic damping case, to see that $S \sim t^{-2}$ as $t \rightarrow \infty$ and that $S \sim t^{-1}$ as $t \rightarrow \infty$ in the cubic case of $p=4$.  Hence, for large enough time, the energy decays like a rational rather than exponential function. We then conclude that all numerical
solutions converge to the same attracting solution for large times, independent
of the initial condition. 

\textit{Sublinear growth:} If $p < 2$, we can simply invert $g(v)$ for $|v| \leq 1$ to get $v = |g(v)|^{q-2}g(v)$, where $q$ is the conjugate exponent for $p$, namely $q = p/(p-1)$.
Hence it suffices to choose $J(s) = 2^{3-4/q}s^{2/q}$.
The ODE \eqref{eq:ode-final} is the same as \eqref{eq:ode-superlinear} with $p$ replaced by $q$ (note that $q > 2$).

\subsection{Difference estimates}
\label{ss:cd}
We again consider solutions $\phi_{1,h},\phi_{2,h}$ corresponding to
different source terms, $F_1,F_2$ as well as different initial conditions.
Once again we define $F = F_1-F_2$, and $E(t)$ is the energy of the difference $\phi_h = \phi_{1,h}-\phi_{2,h}$.
We assume $E(0) \leq \delta_1$ and $\|F\|^2 = \|F_1-F_2\|^2 \leq
\delta_2$, where $\delta_1,\delta_2 > 0$ are ``small" parameters.
Here, we give continuous dependence results in the form of estimates
on $E(t)$ in terms of $\delta_1$ and $\delta_2$.  Our estimates are
\emph{uniform} in time.

The results in this section require an additional assumption on the
function $J$ arising from Lemma~\ref{lem:auxiliaryfunction2}.
In particular, we assume that there exist constants $C_0 > 0$ and
$\alpha \in (0,1)$ such that
\begin{equation}
  \label{eq:Jalpha}
J(s) \leq C_0 s^\alpha.
\end{equation}
The functions $J$ arising from power-law damping considered in the
above examples all satisfy such an estimate, so the results to follow
still hold for the cases of practical interest.

\begin{theorem}
  \label{thm:cd}
  Suppose that~\eqref{eq:Jalpha} holds.
  Let $\phi_{1,h,t}$ and $\phi_{2,h,t}$ denote solutions
  of~\eqref{eq:secondorder_discrete} corresponding to different
  initial conditions and forcing functions $F_1$ and $F_2$ and let
  $E(t)$ denote the energy of their difference.  Suppose that $E(0) =
  \delta_1$ and $\| F \|^2 \equiv \| F_1 - F_2 \|^2 \leq \delta_2$ for
  all time.  Then there exists
  $C>0$ such that
  \[
  E(t) \leq 3 \left( \delta_1 + C \delta_2^{\alpha/(2-\alpha)} \right)
  \]
  for all $t > 0$.
\end{theorem}
\begin{proof}
We define
\begin{equation}
D(t) = \int_{\Omega} (g(\phi_{1,h,t}(t,x)) - g(\phi_{2,h,t}(t,x)))\cdot \phi_{h,t}(t,x)dx. 
\end{equation}
So the energy identity can be written
\begin{equation}
E'(t) + D(t) = (F,\phi_{h,t}).
\end{equation}
Moreover, by Lemmas \ref{lem:auxiliaryfunction2} and \ref{lem:elementary}, we have
\begin{multline}
|g(v) - g(w)|^2 + |v-w|^2 
 \leq 2M(g(v) - g(w))\cdot (v-w)
 \\ + 2J((g(v) - g(w))\cdot (v-w)) \ \forall x,v,w.
\end{multline}

Fix $\delta > 0$ to be chosen (in terms of $\delta_1,\delta_2$) later on.
Then we have, by Young's inequality,
\begin{equation}
2J(s) \leq (2C_0)^{1/\alpha}\delta^{-1/\alpha}s + \delta^{1/(1-\alpha)}
\end{equation}
and thus
\begin{multline} \label{eq:delta-differences}
|g(v) - g(w)|^2 + |v-w|^2  \\
 \leq (2M+ (2C_0)^{1/\alpha}\delta^{-1/\alpha})(g(v) - g(w))\cdot (v-w) + \delta^{1/\alpha} \ \forall x,t.
\end{multline}

\textit{Step 1.}
We start from \eqref{eq:energybound2} in the previous section.

Now \eqref{eq:energybound2} becomes
\begin{multline}
\label{eq:differences1}
\frac{d}{dt}\left(\phi_{h,t},\phi_h^D \right)_{\frac{1}{H}}
+ \frac{\beta}{2\epsilon^2} \| \nabla \cdot \phi_h \|^2
+ \| \phi_{h,t} \|^2_{\frac{1}{H}}
\\
\leq
C_1\|\phi_{h,t}\|^2 + C_2 \| g(\phi_{1,h,t})-g(\phi_{1,h,t})\|^2 + C_2 \| F \|^2,
\end{multline}
where
\begin{equation*}
  \begin{split}
C_1 := & 2 + \frac{3(f^* C_P )^2}{2\beta H_*} \\
C_2 := & \frac{3(\epsilon C_P )^2 H^*}{2\beta H_*}.
  \end{split}
\end{equation*}
Applying \eqref{eq:delta-differences} to \eqref{eq:differences1}, we get
\begin{equation}
\label{eq:differences2}
\frac{d}{dt}\left(\phi_{h,t},\phi_h^D \right)_{\frac{1}{H}}
+ \frac{\beta}{2\epsilon^2} \| \nabla \cdot \phi_h \|^2
+ \| \phi_{h,t} \|^2_{\frac{1}{H}}
\leq
B(\delta)D(t) + 
C(\delta)
+ C_2\delta_2
\end{equation}
where
\begin{equation}
B(\delta) := (2M+ (2C_0)^{1/\alpha}\delta^{-1/\alpha})(C_1+C_2),
\ \ \
C(\delta) := \delta^{1/(1-\alpha)}|\Omega|(C_1+C_2).
\end{equation}

Now, we have that
\(
D(t)
= \left( F, \phi_{h,t} \right) - \frac{d}{dt} E(t),
\)
so that
\begin{multline}
\label{eq:error5}
\frac{d}{dt}\left[ \left(\phi_{h,t},\phi_h^D \right)_{\frac{1}{H}} + \ B(\delta) E(t) \right]
+ \frac{\beta}{2\epsilon^2} \| \nabla \cdot \phi_h \|^2
+ \| \phi_{h,t} \|^2_{\frac{1}{H}}
\\
\leq
B(\delta) \left(F , \phi_{h,t} \right) + C(\delta)  + C_2\delta_2.
\end{multline}
Then, using Young's inequality with appropriate weighting and dividing through by $B(\delta)$ gives
\begin{equation}
\frac{d}{dt} A(t) + \frac{1}{B(\delta)}E(t) \leq \frac{C(\delta)}{B(\delta)} + \left(\frac{B(\delta)}{2} + \frac{C_2}{B(\delta)}\right)\delta_2,
\end{equation}
where
\begin{equation}
    A(t) :=  E(t) + \frac{1}{B(\delta)} \left( \phi_{h,t} , \phi_h^D \right)_{\frac{1}{H}}.
\end{equation}

We will assume that $\delta$ is small enough so that
\begin{equation}
\delta \leq 2C_0\left(\frac{(C_1+C_2)\sqrt{\beta H_*}}{2C_P \epsilon} \right)^{\alpha},
\end{equation}
which is a sufficient condition to show
\begin{equation}
B(\delta) = (2M+ (2C_0)^{1/\alpha}\delta^{-1/\alpha})(C_1+C_2) \geq \frac{2C_P \epsilon}{\sqrt{\beta H_*}},
\end{equation}
which implies \eqref{eq:equivalence} as before.
(Alternatively, just assume $M$ is large.)
So, $A(t)$ is asymptotically equivalent to the energy, and, from \eqref{eq:error5}, we have the bound
\begin{equation}
   \frac{d}{dt}A(t) + \frac{2}{3B(\delta)}A(t) \leq \frac{C(\delta)}{B(\delta)} + \left(\frac{B(\delta)}{2} + \frac{C_2}{B(\delta)}\right)\delta_2,
\end{equation}
which implies
\begin{multline} \label{eq:error6}
  A(t) \leq e^{-\frac{2t}{3B(\delta)}} A(0)
+ \left[\frac{C(\delta)}{B(\delta)} + \left(\frac{B(\delta)}{2} + \frac{C_2}{B(\delta)}\right)\delta_2\right]\int_0^t e^{-\frac{2t}{3B(\delta)}(s-t)} dt
\\
\leq \frac{3}{2}e^{-\frac{2t}{3B(\delta)}} E(0)
+ \frac{3}{2}\left[C(\delta) + \left(\frac{B(\delta)^2}{2} + C_2\right)\delta_2\right]
\end{multline}
Note that $C(\delta) \to 0$ and $B(\delta) \to \infty$ as $\delta \to 0$.
In order to get an estimate, we need $B(\delta)^2\delta_2 \to 0$ as $\delta_1 \to 0$.
Since as $\delta \to 0$ we have $B(\delta) = O(\delta^{-1/\alpha}),$
we can pick $\delta = \delta_2^r$ for any $r \in (0,\alpha/2)$, 
so that
$$
B(\delta)^2\delta_2 = O(\delta_2^{1-2r/\alpha}), \ \delta_2 \to 0.
$$
On the other hand, $C(\delta) = O(\delta_2^{r/(1-\alpha)})$, so the optimal constant $r$ makes these two exponents equal, namely
$$
r = \frac{\alpha(1-\alpha)}{2-\alpha} \ \Rightarrow \ \frac{r}{1-\alpha} = 1-\frac{2r}{\alpha} = \frac{\alpha}{2-\alpha}.
$$
Then \eqref{eq:error6} implies
\begin{equation}
E(t) \leq 3\delta_1 + 3C_3\delta_2^{\alpha/(2-\alpha)}, \ \delta_1,\delta_2 \to 0.
\end{equation}
where
\begin{equation}
C_3 := C_2 + (C_1+C_2)|\Omega| + (C_1+C_2)^2(2M + (2C_0)^{1/\alpha}).
\end{equation}

Note that we also have a precise characterization of the $C$ in the
theorem statement.
\end{proof}

\section{Error estimates}
\label{se:error}
Now, we consider \emph{a priori} error estimates of two types.  For
one, we give an estimate which is optimal with respect to the power of
$h$ but has a possible exponential increase in time.  This is obtained
by using monotonicity of the damping term but no further techniques.
Second, we can also adapt the continuous dependence results of the
previous section to give an estimate that is \emph{uniform} in time,
but has a suboptimal rate with respect to $h$.  

As is typical, we obtain our results by 
comparing the the finite element solution to the $\Pi$ projection of the
true solution, whence the error estimates follow by the triangle
inequality. 

We define 
\begin{equation}
\begin{split}
\chi &\equiv \Pi u - u, \\
\theta_h & \equiv \Pi u - u_h,  \\
\rho & \equiv \pi \eta - \eta, \\
\zeta_h & \equiv \pi \eta - \eta_h
\end{split}
\end{equation}

The projection $\Pi \phi$ satisfies the second-order equation similar to~\eqref{eq:secondorder_discrete}
\begin{multline}
  \left( \Pi \phi_{tt} , v_h \right)
  + \left( \frac{f}{H\epsilon} \Pi \phi_t^\perp , v_h \right)
  + \frac{\beta}{\epsilon^2} \left( \nabla \cdot \Pi \phi , \nabla
  \cdot v_h \right)
  + \left( g\left( \Pi \phi_t \right) , v_h \right)
\\  =
  \left( F , v_h \right)
  + \left( \chi_{tt} , v_h \right)
  + \left( \frac{f}{H\epsilon} \chi_t^\perp , v_h \right)
  + \left( g\left( \Pi \phi_t \right)
  - g\left( \phi_t \right), v_h \right).
\end{multline}

Subtracting the discrete equation~(\ref{eq:secondorder_discrete}) from this gives
\begin{multline}
  \label{eq:secondorderror}
  \left(\theta_{h,tt} , v_h \right)
  + \left( \frac{f}{H\epsilon} \theta_{h,t}^\perp , v_h \right)
  + \frac{\beta}{\epsilon^2} \left( \nabla \cdot \theta_h , \nabla
  \cdot v_h \right)
  + \left( g\left( \Pi \phi_t \right)
         - g\left( \phi_{h,t} \right) , v_h \right)
\\  =
  \left( \chi_{tt} , v_h \right)
  + \left( \frac{f}{H\epsilon} \chi_t^\perp , v_h \right)
  + \left( g\left( \Pi \phi_t \right)
  - g\left(\left( \phi_t \right), v_h \right) \right),
\end{multline}
and putting $v_h = \theta_{h,t}$ and defining
\begin{equation}
  \label{eq:Edef}
E(t) := \frac{1}{2} \| \theta_{h,t} \|^2_{\frac{1}{H}}
+ \frac{\beta}{2\epsilon^2} \| \nabla \cdot \theta_h \|^2
\end{equation}
gives
\begin{equation}
  \frac{d}{dt} E(t)
  + \left( g\left( \Pi \phi_t \right)
         - g\left( \phi_{h,t} \right) , \theta_{h,t} \right)
=
  \left(\tilde F , \theta_{h,t} \right),
\end{equation}
where
\begin{equation}
\tilde F := \chi_{tt} + \frac{f}{H\epsilon} \chi_t^\perp
+ g\left( \Pi \phi_t \right)
  - g\left( \phi_t \right)
\end{equation}
and the Lipschitz condition for $g$ and approximation estimates for
$\chi$ give
\begin{equation}
\| \tilde F \|
\leq \left( C_\Pi |\phi_{tt}|_k
+ \left( \frac{C_\Pi f^*}{H_* \epsilon} + M \right) | \phi_{t} |_k \right) h^k
:= \kappa(\phi) h^k
\end{equation}

The initial conditions here depend on the choice of initial conditions for the discrete equation.  If they are chosen to be the appropriate $\Pi$ projection of the original initial conditions (i.~e.~the $\Pi$ projection of $\phi$ and the $\frac{1}{H}$-weighted $L^2$ projection of $\phi_{t}$)
then the initial condition for the error equation will vanish.

Simply using monotonicity of $g$ gives
\begin{equation}
  \frac{d}{dt} E(t) \leq \frac{1}{2} \| \tilde{F} \|^2 + \frac{1}{2} E(t),
\end{equation}
and it is easy to show from this that
\begin{equation}
E(t) \leq e^{\frac{t}{2}}E(0) + h^{2k} \int_0^t e^{\frac{t-s}{2}} \kappa(\phi)^2 ds.
\end{equation}
Even supposing that $\kappa(\phi)$ is uniformly bounded in time by
\begin{equation}
  \kappa(\phi) \leq \overline{\kappa}
\end{equation}
and the initial conditions are selected so that $E(0) = 0$, one still has
a bound on $E(t)$ that grows exponentially in time.  Combining this
estimate with the triangle inequality leads to the estimate.

\begin{theorem}
    Suppose that and $E(0) = 0$.  Then for all time we have the error estimate
    \begin{multline}
      \frac{1}{2} \| u(\cdot, t) - u_h(\cdot, t) \|^2
      + \frac{\beta}{2 \epsilon} \| \eta(\cdot, t) - \eta_h(\cdot, t) \|^2
      \\ \leq
      \| \chi(\cdot, t)  \|^2
      + \frac{\beta}{\epsilon} \| \rho(\cdot, t) \|^2 + 2E(t)
      \\ \leq \left[
        C_\Pi^2 |u|^2_k + \frac{\beta}{\epsilon} C_\pi^2 |\eta|^2_k
        + 4 \overline{\kappa}^2 \left( e^{\frac{t}{2} - 1} \right) \right] h^{2k}.
    \end{multline}
\end{theorem}

Now, we can employ the continuous dependence results developed earlier
to remove the exponential dependence in time at the expense of a
somewhat decreased rate in $h$.  Returning to Theorem~\ref{thm:cd}, we set
that $\delta_1 = 0$ (for appropriately chosen discrete initial conditions) and
$\delta_2 = \overline{\kappa}^2 h^{2k}$ to obtain the estimate
\begin{theorem}
  Suppose that $\phi, \phi_t \in L^\infty(0,t;H^k(\Omega)$ for all time $t$, the conditions of Theorem~\ref{thm:cd} hold.  Provided the error energy given by~\eqref{eq:Edef} satisfies $E(0) = 0$, then 
  \begin{equation}
    E(t) \leq 3 C_3 \overline{\kappa}^{\frac{2\alpha}{2-\alpha}} h^{\frac{2k\alpha}{2-\alpha}}
  \end{equation}
  and hence
  \begin{multline}
      \frac{1}{2} \| u(\cdot, t) - u_h(\cdot, t) \|^2
      + \frac{\beta}{2 \epsilon} \| \eta(\cdot, t) - \eta_h(\cdot, t) \|^2
      \\ =
      \frac{1}{2} \| \phi_t(\cdot, t) - \phi_{h,t}(\cdot, t) \|^2
      + \frac{\beta}{2 \epsilon} \| \nabla \cdot \phi(\cdot, t)
      - \nabla \cdot \phi_h(\cdot, t) \|^2
      \\ \leq
      \| \chi(\cdot, t)  \|^2
      + \frac{\beta}{\epsilon} \| \rho(\cdot, t) \|^2 + 2E(t)
      \\ \leq
      \left[
        C_\Pi^2 |u|^2_k + \frac{\beta}{\epsilon} C_\pi^2 |\eta|^2_k
        \right] h^{2k}
      + 6 C_3\overline{\kappa}^{\frac{2\alpha}{2-\alpha}} h^{\frac{2k\alpha}{2-\alpha}}
  \end{multline}
  \label{thm:longtimeerror}
  \end{theorem}
Note that this estimate is necessarily suboptimal since
$\alpha \in (0,1)$.  
In the case of a superlinear power law, we have $\alpha = \frac{2}{p}$
for $p > 2$.  For the quadratic damping case with $p = 3$, we have
$\frac{\alpha}{2-\alpha} = \frac{1}{2}$ and hence we have an estimate
on the order of $h^{\frac{k}{2}}$, or $\sqrt{h}$ in the case of the
lowest-order method.  For the cubic power law, this becomes
$\sqrt[3]{h}$.  We do not claim that the present estimates are sharp,
but we are unaware of other techniques to give estimates holding
uniformly in time.

\section{Numerical results}
\label{sec:num}
In this section we present numerical experiments illustrating the
preceding theory.  All numerical results are obtained using the
open-source Firedrake package~\cite{rathgeber2016firedrake}, an
automated solution for the solution of partial differential equations.
We have used Crank-Nicholson with $\Delta t = 0.5 h$, where $h$ is the
characteristic mesh size, and lowest-order
Ravariat-Thomas spaces for most of the simulations in this paper,
although we do consider the convergence rate for the next-to-lowest
order method as well.

In all of our cases, we consider $\epsilon = \beta = 0.1$ and $f=0$.
We consider the linear damping model $g(u) = Cu$ with $C=10$ as
in~\cite{CoKi}.  Additionally, we consider quadratic damping with $g(u) =
C|u| u$ and cubic damping with $g(u) = C|u|^2 u$, also with $C=10$.



\subsection{Damping rates and synchronization}
Now, we demonstrate numerically the effect of the damping function
$g$ on the rate at which energy decays in an unforced system.  We
consider the unit square and a random initial condition with unit
energy and such that $\eta$ has zero mean and run the unforced system on a
\(20 \times 20\) mesh divided into right triangles until $T=100$.  We show
the results of damping in Figure~\ref{fig:damping}.  Each curve shows
the (eventual) rates indicated in our earlier theory, although a
small nonzero energy remains after some time in the linear case.
We also, for a fixed initial condition, reran the simulations with
decreasing time step, and observed that this residual energy decreases
like $\mathcal{O}(\Delta t^2)$, so it is likely an artifact of the
time discretization.

\begin{figure}
  \centerline{\includegraphics[width=8cm]{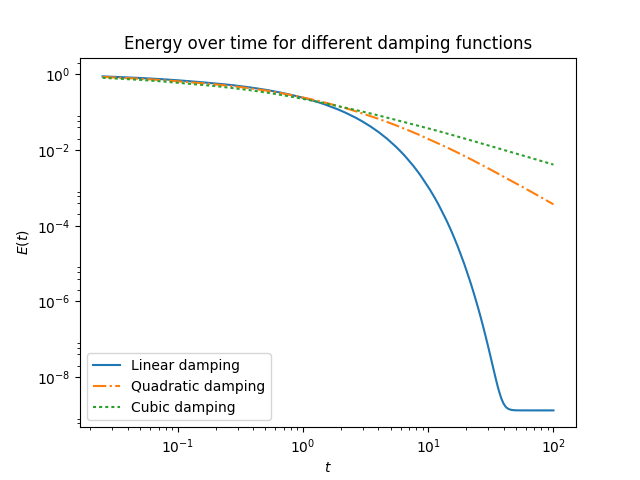}}
  \caption{Damping rates, starting from random initial data.  After a
    start-up period, the energy decays like $1/t$ for the cubic
    damping, $1/t^2$ for quadratic, and exponentially for the linear
    damping.}
  \label{fig:damping}
\end{figure}

In Figure~\ref{fig:sync}, we consider the case of two distinct random
initial conditions, subjecting both the the
forcing $(F,v) = \frac{\beta}{\epsilon^2} \sin(t) \left(xy, \nabla
\cdot v\right)$ and measure the energy of the differences between
solutions over time $[0, 100]$.  As with the damping, the eventual
observed rates match those predicted theory, although
there is a residual energy like in the damping example.

\begin{figure}
  \centerline{\includegraphics[width=8cm]{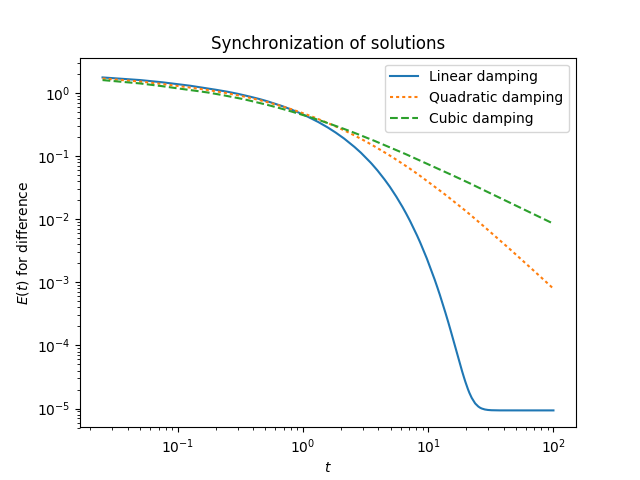}}
  \caption{Synchronization rates, starting from random initial data.  After a
    start-up period, the energy of the difference decays like $1/t$
    for the cubic damping, $1/t^2$ for quadratic, and
    exponentially for the linear damping until a residual energy of
    $\mathcal{O}(\Delta t^2)$ is reached.}
  \label{fig:sync}
\end{figure}

\subsection{Convergence}
We used the method of manufactured solutions on the unit square, setting the problem coefficients to unit value and choosing forcing functions to make
\begin{equation}
  \begin{split}
    u(x, y, t) & = \cos(\pi t) \begin{bmatrix}
      \sin(\pi x) \cos (\pi y), &
      \cos(\pi x) \sin(\pi y) \end{bmatrix}^T \\
    \eta(x, y, t) & = \sin(\pi x) \sin(2 \pi y) \cos(\pi t)
  \end{split}
\end{equation}

In each case, we computed the solution until $T=10$, measuring the
$L^2$ error of both $u$ and $\eta$ at each time step.
We observed full first-order convergence, which is
predicted in the linear case.  In the nonlinear cases, it is not clear
whether Theorem~\ref{thm:longtimeerror} is suboptimal, the calculation has not
run on a long enough time horizon, or there is some other
consideration.  At any rate, we have confirmed similar second-order
convergence when using the next-to-lowest order Raviart-Thomas
element.

\section{Conclusions}
In this paper we introduced several results that underpin the
application of compatible finite element spaces to barotropic tide
modelling with nonlinear drag terms that are used in barotropic global
tide models. By importing results from nonlinear PDEs, we were able to
show that the numerical discretisation has a global attracting
solution. Calculating this solution is the goal of barotropic tide
modelling, since the Earth's tides are assumed to have being occuring
on a long enough time scale that memory of the initial conditions or
past changes in topography are not relevant. The proof requires some
assumptions of linear growth at infinity of the damping term, but in
practice any reasonable damping model can be adjusted a posteriori to
have linear growth at values that are never attained in the model
solutions. We then provided two error analyses. The first is over
finite time intervals, and predicts optimal scaling with mesh
resolution, but with constant of proportionality growing exponentially
in time. The second is global in time, but we obtain a suboptimal
scaling with mesh resolution; numerical experiments confirm that our
estimate is not sharp. However, our analysis does show that the
numerical attracting solution converges to the true attracting solution
as the mesh is refined, i.~e.~we have a convergent numerical solution
to the barotropic tidal prediction problem. 

\bibliographystyle{plain}
\bibliography{bib}

\end{document}